\documentclass[a4paper,12pt]{article}
\usepackage{amsmath}
\usepackage{amssymb}
\usepackage{setspace}
\usepackage{fullpage}
\usepackage{pgf,tikz}
\usetikzlibrary{arrows}
\usepackage{bbm}
\usepackage{booktabs}
\usepackage{amsthm}
\usepackage[latin1]{inputenc}
\usepackage{xcolor}
\usepackage{hyperref}
\usepackage{pbox}
\usepackage{graphicx}
\newtheorem{theorem}{Theorem}[section]
\newtheorem{lemma}[theorem]{Lemma}
\newtheorem*{rtheorem}{Main Theorem}
\newtheorem{corollary}[theorem]{Corollary}

\theoremstyle{definition}

\newtheorem*{definition}{Definition}

\usepackage{pdfsync}
\def\PG{\mathrm{PG}}
\def\AG{\mathrm{AG}}
\def\F{\mathbb{F}}
\def\D{\mathcal{D}}
\def\L{\mathcal{L}}
\def\S{\mathcal{S}}
\def\e{\epsilon}
\def\l{\ell}

\title{Unitals with many Baer secants through a fixed point}

\author{Sara Rottey\thanks{VUB, Department of Mathematics, Pleinlaan 2, 1050 Brussel, Belgium.\newline Email: srottey@vub.ac.be}\and Geertrui Van de Voorde \thanks{This author is a postdoctoral fellow of the Research Foundation Flanders (FWO -- Vlaanderen).\newline UGent, Department of Mathematics, Krijgslaan 281 -- S22, 9000 Gent, Belgium.\newline Email: gvdvoorde@cage.ugent.be}}

\begin{document}
\maketitle
\begin{abstract} We show that a unital $U$ in $\PG(2,q^2)$ containing a point $P$, such that at least $q^2-\e$ of the secant lines through $P$ intersect $U$ in a Baer subline, is an ovoidal Buekenhout-Metz unital (where $\e\approx 2q$ for $q$ even and $\e\approx q^{3/2}/2$ for $q$ odd).\end{abstract}
\section{Introduction}
This paper studies unitals in the Desarguesian projective plane of square order $q^2$, $q=p^h$, $p$ prime, denoted by $\PG(2, q^2)$. A {\em unital} $U$ in $\PG(2, q^2)$ is a set of $q^3 + 1$ points of $\PG(2, q^2)$ such that each line contains exactly $1$ or $q + 1$ points of $U$.

An example of a unital in $\PG(2, q^2)$ is given by the set of absolute points of
a unitary polarity, called a {\em classical unital}.  In \cite{BuekenhoutUnitals}, Buekenhout constructed a class of unitals, called {\em ovoidal Buekenhout-Metz unitals} (see Subsection \ref{sectionunital}). Every known unital can be obtained by this construction.

Combining the results of \cite{casse} (for $q$ odd and $q>3$), and \cite{Okeefe} (for $q>2$ even and $q=3$) the following characterisation of ovoidal Buekenhout-Metz unitals is obtained.

\begin{theorem}\label{alle secanten}{\rm \cite{casse,Okeefe}} Let $U$ be a unital in $\PG(2,q^2)$, $q>2$, containing a point $P$ such that all secants through $P$ intersect $U$ in a Baer subline, then $U$ is an ovoidal Buekenhout-Metz unital with special point $P$.
\end{theorem}

Moreover, two related characterisations were found for classical unitals.
\begin{theorem}{\rm\cite{Barwick3}}\label{Barwick} Let $U$ be an ovoidal Buekenhout-Metz unital in $\PG(2,q^2)$ with special point $P$. If $U$ contains a secant not through $P$ intersecting $U$ in a Baer subline, then $U$ is classical.
\end{theorem}

\begin{theorem}{\rm \cite{BallBlokhuisOkeefe}}
Let $U$ be a unital in $\PG(2, p^2)$, $p$ prime, such that $p(p^2-2)$ secants
intersect $U$ in a Baer subline, then $U$ is classical.
\end{theorem}
Concerning these three results, in \cite[Open problems 4]{CurrentResearchTopics} the following question was posed:

\begin{minipage}[c][4em][c]{14cm}
\begin{itemize}
\item[] What is the minimum required number of secants being Baer sublines, to conclude that a unital is an ovoidal Beukenhout-Metz unital?
\end{itemize}
\end{minipage}

In this paper, we will improve the result of Theorem \ref{alle secanten}, and find a new upper bound for the minimum required number of Baer secants through a fixed point of the unital. It is worth noticing that our Main Theorem implies the result of \cite{casse} and \cite{Okeefe} for $q\geq 16$.
\begin{rtheorem}Suppose $q$ and $\e$ satisfy the conditions of Table \ref{tableunitals}. Let $U$ be a unital in $\PG(2,q^2)$ containing a point $P$ such that at least $q^2-\e$ of the secants through $P$ intersect $U$ in a Baer subline, then $U$ is an ovoidal Buekenhout-Metz unital with special point $P$.
\end{rtheorem}

\begin{table}[!h]
\begin{center}\def\arraystretch{1.7}
\begin{tabular}{llll}\toprule
$\e$ &&& Conditions \\
\midrule
$\e \leq q-3$&&&$q$ even, $q\geq 16$\\
$\e \leq 2q-7$ &&&$q$ even, $q\geq 128$ \\
&&&\\
$\e \leq \frac{\sqrt{q}q}{4} - \frac{39q}{64} -O(\sqrt{q})+1$ &&& $q$ odd, $q\geq 17$, $q=p^{2e}$, $e\geq 1$ \\
$\e \leq  \frac{\sqrt{q}q}{2} - {2q}$ &&& $q$ odd, $q\geq 17$, $q=p^{2e+1}$, $e\geq 1$ or $q$ prime \\
$\e \leq  \frac{\sqrt{q}q}{2} - \frac{67q}{16} +\frac{5\sqrt{q}}{4}-\frac{1}{12}$ &&& $q$ odd, $q\geq 17$, $q=p^h$, $p\geq5$ \\
$\e \leq \frac{\sqrt{q}q}{2}-\frac{35q}{16} - O(\sqrt{q})+1$ &&& $q$ odd, $q\geq23^2$, $q=p^h$, $h$ even for $p=3$, $q\neq 5^5, 3^6$\\
\bottomrule
\end{tabular}
\caption{Conditions for the Main Theorem}\label{tableunitals}
\end{center}
\end{table}

\section{Preliminaries}

\subsection{Sublines and subplanes in the ABB-representation}

A {\em Baer subline} in $\PG(1,q^2)$ is a set of $q+1$ points whose coordinates, with respect to three well-chosen distinct points, are in $\F_q$. Similarly, a {\em Baer subplane} of $\PG(2,q^2)$ is a set of $q^2+q+1$ points, whose coordinates with respect to a well-chosen frame, are contained in $\F_q$. The $q^2+q+1$ points of a Baer subplane in $\PG(2,q^2)$, together with the lines they induce, form a projective plane of order $q$.

Andr\'e \cite{Andre} and Bruck and Bose \cite{Bruck-Bose} independently found a representation of
translation planes of order $q^n$, with kernel containing $\F_q$, in the projective
space $\PG(2n, q)$.
We refer to this as the {\it Andr\'e/Bruck-Bose-representation} or the {\it ABB-representation}.
In this paper, we restrict ourselves to the case $n=2$.

Let $\S$ be a line spread in $\PG(3,q)$. Embed $\PG(3, q)$ as a hyperplane $H_\infty$ in $\PG(4, q)$. Consider the following incidence structure $A(\S)$ with point set $P$ and line set $L$, where incidence is natural:
\begin{itemize}
\item[$P:$] the affine points, i.e. the points of $\PG(4, q) \backslash H_\infty$,
\item[$L:$] the planes of $\PG(4,q)$ intersecting $H_\infty$ exactly in a line of $\S$.
\end{itemize}
In \cite{Bruck-Bose} the authors showed that $A(\S)$ is an affine translation plane of order $q^2$, and conversely, every such translation plane can be constructed in this way. The plane $A(\S)$ is a Desarguesian affine plane $\AG(2, q^2)$ if and only if the spread $\S$ is Desarguesian.
The {\em projective completion} $\overline{A(\S)}$ of the affine plane $A(\S)$ can be found by adding $H_\infty$ as the line $\l_\infty$ at infinity, where the lines of $\S$ correspond to the points of $\l_\infty$. Clearly, the projective completion $\overline{A(\S)}$ is a Desarguesian projective plane $\PG(2, q^2)$ if and only if the spread $\S$ is Desarguesian.

In this paper, we will fix a line $\ell_\infty$ at infinity of $\PG(2,q^2)$ and consider the ABB-representation with respect to this line. The hyperplane at infinity of $\PG(4,q)$, corresponding to $\l_\infty$, will be denoted by $H_\infty$ and the Desarguesian spread defining $\PG(2,q^2)$ by $\D$.

We will call a Baer subline {\em tangent} (to $\ell_\infty$) if it has one point in common with $\ell_\infty$, and {\em external } if it has no such intersection point.
In the ABB-representation, tangent sublines of $\PG(2,q^2)$ are in one-to-one correspondence with lines of $\PG(4,q)$ intersecting $H_\infty$ in exactly one point. An external subline corresponds to a non-degenerate conic of $\PG(4,q)$, called a {\em Baer conic}, contained in a plane which meets $H_\infty$ in a spread line of $\D$, external to this conic. Note that, unless $q=2$, not every conic is a Baer conic. Moreover, since any two distinct Baer sublines have at most two points in common, we also have that any two distinct Baer conics share at most 2 points.

A Baer subplane will be called {\em secant} (to $\ell_\infty$) if it meets $\l_\infty$ in $q+1$ points, and {\em tangent} if it meets $\ell_\infty$ in one point.
In the ABB-representation, secant subplanes are in one-to-one correspondence with  planes of $\PG(4,q)$ intersecting $H_\infty$ in a line not contained in $\D$.
A tangent Baer subplane corresponds to the point set of $q+1$ disjoint lines, called {\em generator lines}, forming a ruled cubic surface, called a {\em Baer ruled cubic}. Such a Baer ruled cubic has a spread line $T\in\D$ as {\em line directrix}, where $T$ is the line corresponding to the intersection point of the tangent Baer subplane with $\l_\infty$. As a base it has a Baer conic $C$ in a plane disjoint from $T$. For each point of $T$, there is a unique generator line on the Baer ruled cubic through this point and a point of $C$. A plane through a line of $\D\setminus\{T\}$ intersects the Baer ruled cubic in a point or a Baer conic.
For more information on the ABB-representation of sublines and subplanes of $\PG(2,q^2)$, we refer to {\rm \cite{Ebert}}.

It is well-known that two distinct Baer sublines of $\PG(2,q^2)$, that have a common point, are contained in a unique Baer subplane. The following lemma, in terms of lines of $\PG(4,q)$ in the ABB-representation, can be deduced.

\begin{lemma}\label{uniquesecantplane}
Two lines of $\PG(4,q)$, not contained in a plane through a line of $\D$, intersecting $H_\infty$ in the same point, lie on a unique plane intersecting $H_\infty$ not in a line of $\D$, i.e. they define a unique secant subplane to $\ell_\infty$.

Two lines of $\PG(4,q)$, not in $H_\infty$, through different points $P_1, P_2$ of $H_\infty$, such that $P_1 P_2$ is a spread line of $\D$, lie on a unique Baer ruled cubic, i.e. they define a unique tangent subplane to $\ell_\infty$.
\end{lemma}

\subsection{Caps and ovoids in \texorpdfstring{$\PG(3,q)$}{PG(3,q)}}

A {\em $k$-cap} in $\PG(3, q)$ is a set of $k$ points no three of which are collinear. A $k$-cap is called {\em complete} if it is not contained in a $(k+1)$-cap of $\PG(3,q)$. It is well-known that a cap of $\PG(3,q)$, $q>2$, contains at most $q^2+1$ points; a $(q^2+1)$-cap is also called an {\em ovoid}.

We will need the following extendability results for caps in $\PG(3,q)$.
\begin{theorem} \label{cap}
A cap in $\PG(3,q)$, of size at least $q^2-\delta$, with $\delta$ and $q$ satisfying the conditions of Table \ref{tablecaps}, can be extended to an ovoid.
\end{theorem}

\begin{table}[!h]
\begin{center}\def\arraystretch{1.7}
\begin{tabular}{lllllll}\toprule
$\delta$ &&&Conditions &&&Reference \\
\midrule
$\delta \leq \frac{q}{2}+\frac{\sqrt{q}}{2}-1$&&&$q$ even, $q>2$ &&& {\rm \cite{LinearIndependenceInFiniteSpaces}}\\
$\delta \leq q-4$&&&$q$ even, $q\geq 8$ &&& {\rm \cite{Chao}}\\
$\delta \leq 2q-8$ &&&$q$ even, $q\geq 128$ &&&{\rm \cite{Cao}}\\
&&&&&&\\
$\delta \leq \frac{\sqrt{q}q}{4} - \frac{39q}{64} -O(\sqrt{q})$ &&& $q$ odd, $q\geq 17$, $q=p^{2e}$, $e\geq 1$ &&& {\rm\cite{PackingProblem}}\\
$\delta \leq  \frac{p^{e+1}q}{4} -\frac{119pq}{64} + O(p^{e+2})$ &&& $q$ odd, $q\geq 17$, $q=p^{2e+1}$, $e\geq 1$ &&&{\rm\cite{PackingProblem}}\\
$\delta  \leq \frac{359q^2}{2700} + \frac{4q}{135} - \frac{94}{27}$ &&& $q$ odd, $q\geq17$ prime&&&{\rm\cite{PackingProblem}}\\
$\delta \leq  \frac{\sqrt{q}q}{2} - \frac{67q}{16} +\frac{5\sqrt{q}}{4}-\frac{13}{12}$ &&& $q$ odd, $q\geq 17$, $q=p^h$, $p\geq5$ &&& {\rm\cite{PackingProblem}}\\
$\delta \leq \frac{\sqrt{q}q}{2}-\frac{35q}{16} - O(\sqrt{q})$ &&& \pbox{20cm}{$q$ odd, $q\geq23^2$, $q=p^h$ \\\hspace*{0.3cm}($h$ even for $p=3$), $q\neq 5^5, 3^6$}&&&{\rm\cite{PackingProblem}}\\
\bottomrule
\end{tabular}
\caption{Conditions for Theorem \ref{cap}}\label{tablecaps}
\end{center}
\end{table}

Moreover, the following theorem shows that the ovoids obtained in the previous theorem are unique.
\begin{theorem}{\rm\cite[Theorem 2.2]{Storme}} \label{uniquecap}
\begin{enumerate}
\item[] If $K$ is a $k$-cap in $\PG(n, q)$, $n\geq 3$, $q$ even, having size $k>(q^{n-1}+\cdots+q+2)/2$, then $K$ can be extended in a unique way to a complete cap.
\item[] If $K$ is a $k$-cap in $\PG(n, q)$, $n\geq  3$, $q$ odd, of size $k> 2(q^{n-1}+\cdots+q +2)/3$, then $K$ can be extended in a unique way to a complete cap.
\end{enumerate}
\end{theorem}

\subsection{Unitals in \texorpdfstring{$\PG(2,q^2)$}{PG(2,q2)}}\label{sectionunital}

Recall that a {\em unital} in $\PG(2,q^2)$ is a set of $q^3+1$ points such that every line meets $U$ in $1$ or $q+1$ points. It is easy to see that a point $P$ of $U$ lies on exactly one tangent line to $U$ and on $q^2$ lines meeting $U$ in $q+1$ points (including $P$). These last lines are called the {\em $(q+1)$-secants}, or short {\em secants}, to $U$. If a secant line meets a unital in a Baer subline, then we call this line a {\em Baer secant}.

A {\em classical unital} (or {\em Hermitian curve}) in $\PG(2, q^2)$ corresponds to the set of absolute points of
a unitary polarity.
Note that every unital in $\PG(2,4)$ is classical. In $\PG(2,q^2)$, $q>2$, there are examples of non-classical unitals.

An {\em ovoidal Buekenhout-Metz unital} in $\PG(2,q^2)$ arises from the following construction (see \cite{BuekenhoutUnitals}).
Consider the ABB-representation in $\PG(4,q)$ of $\PG(2,q^2)$ with respect to the line $\l_\infty$, with line spread $\D$ of $H_\infty$ corresponding to the points of $\l_\infty$. Let $\mathcal{O}$ be an ovoid in $\PG(4, q)$ intersecting $H_\infty$ in a unique point $A$, such that the
tangent plane of $\mathcal{O}$ at $A$ does not contain the spread line $T\in\D$ incident with $A$. Let $V$ be
a point on $T$, $V\neq A$. Consider the ovoidal cone with vertex $V$ and base $\mathcal{O}$, this point set forms a unital $U$ in $\PG(2,q^2)$. The line $\l_\infty$ is the tangent line to $U$ at the point $P_\infty$ of $\l_\infty$, where $P_\infty$ is the point corresponding to the spread line $T$. We will call $P_\infty$ the {\em special point} of the ovoidal Buekenhout-Metz unital $U$.
Clearly, all secants to $U$ at $P_\infty$ are Baer secants.

All known unitals in $\PG(2,q^2)$, including the classical unital, arise as ovoidal Buekenhout-Metz unitals.

\section{Unitals with a point lying on many Baer secants}

In this section, we will prove our main theorem.
We will need the following lemma which can be shown by a simple counting argument.
\begin{lemma}\label{telling}{\rm \cite[Theorem 2.1]{Okeefe}}
A tangent Baer subplane meets a unital in $\PG(2,q^2)$ in at most $2q+2$ points, a secant Baer subplane meets a unital in $\PG(2,q^2)$ in at most $2q+1$ points.
\end{lemma}

Throughout this paper, we will use the following notations and conventions for a given unital $U$ of $\PG(2,q^2)$.

Let $U$ be a unital in $\PG(2,q^2)$ containing a point $P_\infty$ such that a set of at least $q^2-\e$, $\e\leq q^2$, of the $(q+1)$-secants through $P_\infty$ are Baer secants.
Say $\l_\infty$ is the tangent line of $U$ at $P_\infty$ and consider the ABB-representation of $\PG(2,q^2)$, where the points of $\l_\infty$ correspond to the Desarguesian spread $\D$ of the hyperplane $H_\infty$ of $\PG(4,q)$. By abuse of notation, we will use the notation $U_{\rm aff}$ for both the points of $U \setminus \{P_\infty\}$ in $\PG(2,q^2)$ and for the corresponding affine point set in $\PG(4,q)$.

Suppose $P_\infty$ corresponds to the spread line $T$ of $\D$.
Let $\L$ be the set of $q^2-\e$ lines in $\PG(4,q)$ corresponding to Baer secants through $P_\infty$. Every line of $\L$ intersects $H_\infty$ in a point of $T$. Note that any plane intersecting $H_\infty$ in $T$ contains exactly $q$ points of $U_{\rm aff}$.

Given a unital $U$ and its corresponding line set $\L$, we will consider a set $S(U)$ in the plane $\Pi=\PG(4,q)/T$, consisting of points with labels, induced by the lines of $\L$. This point set is defined as follows.
\begin{definition}
Consider the quotient space $\Pi=\PG(4,q)/T$, isomorphic to $\PG(2,q)$, and let $v_1,\ldots,v_{q+1}$ be the points of $T$.
The points of $S(U)$ are the points of $\Pi$ corresponding to the planes through $T$ which contain a line of $\L$.
We label a point $R$ of $S(U)$ with $v_j$, if the line of $\L$ in the plane $\langle T, R\rangle$ goes through $v_j$.
\end{definition}
\begin{lemma} \label{SU}
The set $S(U)$ is a point set in $\AG(2,q)$ such that each point has exactly one label. Moreover, $S(U)$ has the property that if a point $Q$ of $S(U)$ lies on a line of $\AG(2,q)$ containing two points of $S(U)$ with the same label $v$, then $Q$ also has label $v$.
\end{lemma}
\begin{proof} First note that the points of $S(U)$ are contained in an affine plane of $\Pi=\PG(4,q)/T$, since $H_\infty / T$ is a line in $\Pi$ and since no plane through $T$ in $H_\infty$ contains a line of $\L$. Each point of $S(U)$ has exactly one label, as a plane through $T$ contains at most one line of $\L$. If a line $m$ in $\Pi$ contains two points of $S(U)$ with the same label, say $v_k$, then the $3$-space $\langle T,m\rangle$ contains two lines $\l_{1},\l_{2}$ of $\L$ through the point $v_k$. Suppose that there is a point of $S(U)$ on the line $m$ with label $v_j$, $j\neq k$. This implies that there is a line of $\L$, say $\l_{3}$, through $v_j$, contained in $\langle T,m\rangle$. Thus, the line $\l_{3}$ meets the plane $\langle \l_{1},\l_{2}\rangle$ in an affine point, which means that the secant subplane defined by $\l_{1},\l_{2}$ contains $2q+2$ points, a contradiction by Lemma \ref{telling}.
\end{proof}

Next, we show that the configuration of points of $S(U)$ must satisfy one of three conditions.

\begin{lemma} \label{S} Suppose $q>2$ and $k\in\mathbb{N}$, $k<\sqrt{q}-1$. Let $S$ be a set of $q^2-\e$, $\e\leq kq$, points in $\AG(2,q)$, and consider a set of labels $\mathcal{V}=\{v_1,\ldots,v_{q+1}\}$, such that each point of $S$ has exactly one label. Denote the subset of $S$ containing all points with label $v$ by $S_{v}$.

Suppose that the set $S$ has the property that if a point $Q$ of $S$ lies on a line of $\AG(2,q)$ containing two points of $S$ with the same label $v$, then $Q$ also has label $v$. Then the set $S$ satisfies one of the following.
\begin{itemize}
\item[(i)] All points of $S$ have the same label.
\item[(ii)] There are $2$ distinct labels $v_1$ and $v_2$ each occurring at least $q-k$ times as labels of points of $S$. For $i=1,2$, the points of $S_{v_i}$ lie on an affine line. These two affine lines go through a common affine point.
\item[(iii)] There is a subset $\mathcal{V}^*\subseteq\mathcal{V}$ of labels, each occurring at least twice, such that for every label $v \in \mathcal{V}^*$, the points of $S_{v}$ lie on an affine line. These affine lines are all parallel (i.e. their projective completions go through a common point $Q_\infty$ at infinity). The subset $S^*\subseteq S$, consisting of points with a label in $\mathcal{V}^*$, has size at least $q^2-\e-(k^2+k)(k^2+k-1)-1$.

\end{itemize}
\end{lemma}
\begin{proof}
First, make the following two observations.
\begin{enumerate}
\item[$\bullet$]
Suppose that there is a label $v$ appearing $q+2$ times or more. Take a point $P\in S$, then at least one line through $P$ contains at least two points of $S$ with label $v$. Hence, the point $P$ also has label $v$, thus, all points of $S$ have label $v$. We find that $S$ has configuration $(i)$.
\item[$\bullet$]
Suppose that there is a label $v$, such that $q$ points of $S_v$ lie on a line $L$. If $S$ does not have configuration $(i)$, then one can check that no other point of $S$ has label $v$. Moreover, if another label appears at least two times, then the line spanned by the corresponding points must be parallel to $L$. Hence, any label appears at most $q$ times. There is a subset $\mathcal{V}^* \subseteq \mathcal{V}$ containing at least $q-k$ labels, such that every label appears at least twice; otherwise, there would be at most $(q-k-1)q+(k+2)1=q^2-kq-q+k+2<q^2-kq$ points in $S$. There are at most $k+1$ points having a label appearing only once. The subset $S^* \subseteq S$ of points having a label in $\mathcal{V}^*$ has size at least $q^2-\e-k-1 \geq q^2-\e-(k^2+k)(k^2+k-1)-1$. Hence, $S$ has the configuration described in $(iii)$.
\end{enumerate}

Now, consider a label $v$ occurring at least $q-k$ times.
Suppose that there are three non-collinear points in  $S_v$.
Choose a point $P_1 \in S_v$ and consider the set $Z$ of all lines containing $P_1$ and another point of $S_v$. Every line of $Z$ can only contain points with label $v$. Consider the set $Z' \subseteq Z$ of all lines of $Z$ that contain at most $k$ points of $S$ different from $P_1$; suppose $|Z'|=x$. Hence, the lines of $Z'$ each contain at least $q-k-1$ affine points not in $S$. Since the lines of $Z'$ contain at most all $kq$ points not in $S$, we see that
$$ x \leq \frac{kq}{q-k-1}.$$
However, the upper bound for the number of points of $S_v$, different from $P_1$, covered by the lines of $Z'$ is equal to $xk$. We see that
$$xk \leq \frac{k^2q}{q-k-1}.$$
Moreover, when $k<\sqrt{q}-1$, we have
$$\frac{k^2q}{q-k-1} < q-k-1.$$
As there are at least $q-k-1$ points in $S_v$, different from $P_1$, there exists a point $P_2 \in S_v$ not on a line of $Z'$. Hence, the line $P_1P_2$ contains at least $k+1$ points of $S$, different from $P_1$.

Consider a point $P_3\in S_v$, but not on $P_1 P_2$. There are at least $k+2$ lines through $P_3$ and a point of $S\cap P_1 P_2$ containing only points of $S$ with label $v$. These lines cover at least $1+(k+2)(q-1)-kq=2q-k-1\geq q+2$ points of $S$, when $k< \sqrt{q}-1$ and $q>2$. Since the label $v$ appears at least $q+2$ times, it follows that all points of $S$ have label $v$, hence, $S$ has configuration $(i)$.

We can now assume that if a label $v$ appears at least $q-k$ times, then the points of $S_v$ lie on a line. Moreover, since $q$ points with a fixed label on a line imply configuration $(i)$ or $(iii)$, we can pose that 
$\forall v \in \mathcal{V}: |S_v|<q$.
We can count that there are at least two labels $v_1$ and $v_2$ each occurring at least $q-k$ times, since otherwise there would be at most $1(q-1)+q(q-k-1)=q^2-kq-1<q^2-kq$ points in $S$.
Consider the lines $L_1$ and $L_2$ containing all points of $S_{v_1}$ and $S_{v_2}$ respectively.

If $L_1$ and $L_2$ intersect in an affine point $Q$, then $S$ has configuration $(ii)$.

Now, suppose $L_1$ and $L_2$ are parallel, i.e. their projective completions intersect in a point $Q_\infty$ at infinity.
There are at least $q-k+1$ labels occurring at least twice, since otherwise there would be at most $(q-k)(q-1)+(k+1)1=q^2-kq-q+2k+1<q^2-kq$ points in $S$.
A line spanned by two points with the same label (different from $v_1$ and $v_2$) must intersect both lines $L_i$ in a point not in $S$.
However, the line $L_i$, $i=1,2$, contains at most $k$ affine points not in $S$. Hence, there are at most $k^2$ lines intersecting both lines $L_i$, $i=1,2$, not in $Q_\infty$ and not in a point of $S$.
This means that, of all the labels appearing at least twice, there are at most $k^2$ labels such that two points with the same label do not necessarily span a line containing $Q_\infty$.
Hence, there is a subset $\mathcal{V}^*\subseteq\mathcal{V}$ of at least $q-k^2-k+1$ labels occurring at least twice such that points with the same label do lie on a line containing $Q_\infty$.

It follows that there are at most $k^2+k-1$ affine lines through $Q_\infty$, such that the points of $S$ on such a line do not have the same label. However, there are at most $(q+1)-(q-k^2-k+1)=k^2+k$ labels that could occur this way. Hence, at most $(k^2+k-1)(k^2+k)$ points of $S$ have the property that a line spanned by two points with the same label does not necessarily contain $Q_\infty$. It follows that there is a subset $S^*\subseteq S$ of at least $q^2-\e-(k^2+k)(k^2+k-1) > q^2-\e-(k^2+k)(k^2+k-1)-1$ points, having the property that a line spanned by two points with the same label does contain $Q_\infty$, i.e. they have a label in $\mathcal{V}^*$. This means that $S$ has configuration $(iii)$.
\end{proof}

The following three lemmas will show that the affine point set $S(U)$, defined by the unital $U$, must satisfy the first configuration of Lemma \ref{S}.

The subset of $S(U)$ containing all points with label $v_i$, will be denoted by $S_{v_i}(U)$.
\begin{lemma} \label{geendrie} Suppose $q>2$ and $k\in\mathbb{N}$, $k<\sqrt{q}-1$. Let $U$ be a unital containing  a point $P_\infty$ such that $q^2-\e$, $\e\leq kq$, of the $(q+1)$-secants through $P_\infty$ are Baer secants. The corresponding point set $S(U)$ cannot have the form $(ii)$ of Lemma \ref{S}.
\end{lemma}
\begin{proof}
Suppose that $S(U)$ is of the form $(ii)$ of Lemma \ref{S}. There are two distinct labels, say $v_1$ and $v_2$, occurring at least $q-k$ times, such that for $i=1,2$, the points of $S_{v_i}(U)$ lie on an affine line $L_i$. The  affine lines $L_1$ and $L_2$ intersect in an affine point $A$.

Let $T$ be the spread line corresponding to $P_\infty$.
A line of $\L$ through $v_1$ induces a point of $L_1$ in the quotient space $\PG(4,q)/T$. Hence, all the lines of $\L$ containing $v_1$ are contained in the three-space $\Sigma_1=\langle T, L_1\rangle$. Similarly, the lines of $\L$ containing $v_2$ are contained in the three-space $\Sigma_2= \langle T, L_2 \rangle$. Let $\alpha$ be the plane $\langle T,A\rangle$, then clearly $\alpha$ is the intersection $\Sigma_1 \cap \Sigma_2$.
Moreover, as the plane $\alpha$ is not contained in $H_\infty$, there are $q$ points of $U_{\rm aff}$ contained in $\alpha$.

There are at most $k+1$ lines, say $n_1,\ldots,n_{k+1}$, of $\alpha$ through $v_1$ which do not occur as the intersection $\langle \l_i,\l_j\rangle \cap \alpha$, where $\l_i,\l_j$ are lines of $\L$ through $v_1$  in the three-space $\Sigma_1$. Similarly, there are at most $k+1$ lines $n'_1,\ldots,n'_{k+1}$ of $\alpha$ through $v_2$ which do not occur as the intersection $\langle \l_i,\l_j\rangle \cap \alpha$, where $\l_i,\l_j$ are lines of $\L$ through $v_2$  in the three-space $\Sigma_2$.

Suppose that a point of $U$ in $\alpha$ lies on a plane $\langle \l_i,\l_j\rangle$, where $\l_i,\l_j$ are lines of $\L$ through the same point of $T$, then the secant subplane defined by $\l_i,\l_j$ contains $2q+2$ points of $U$, a contradiction by Lemma \ref{telling}. This implies that each of the $q$ points of $U$ in $\alpha$ necessarily lies on one of the lines $n_1,\ldots,n_{k+1}$ and on one of the lines $n'_1,\ldots n'_{k+1}$. However, there are only $(k+1)^2$ such points and $q>(k+1)^2$, a contradiction.
\end{proof}

Consider a Baer subplane $\pi$ of $\PG(2,q^2)$ containing the point $P_\infty$. It is clear that $\pi/P_\infty$ defines a Baer subline in the quotient space $\PG(2,q^2)/P_\infty$. This can be translated to the ABB-representation in the following way.
Recall that a Baer subplane $\pi$, tangent to $\l_\infty$ at $P_\infty$, corresponds to a Baer ruled cubic $\mathcal{B}$ with line directrix $T$. We see that $\mathcal{B}/T$ defines a Baer conic in the quotient space $\PG(4,q)/T$.

\begin{lemma}\label{Baerconic} Suppose $q\geq 16$ and $k\in\mathbb{N}$, $k\leq \sqrt{q}/2-2$. Let $U$ be a unital containing a point $P_\infty$ such that $q^2-\e$, $\e\leq kq$, of the $(q+1)$-secants through $P_\infty$ are Baer secants. Suppose $S(U)$ is as described in Lemma \ref{S} case $(iii)$, with subset $S^*(U)\subseteq S(U)$. Then there exists a Baer ruled cubic $\mathcal{B}$ in $\PG(4,q)$, containing two lines of $\L=\{\l_1,\ldots,\l_{q^2-\e}\}$, such that the corresponding Baer conic in $\PG(4,q)/T$ contains at least $\lfloor\frac{q+7}{2}\rfloor$ points of $S^*(U)$.

\end{lemma}
\begin{proof}
Consider $S(U)$ as described in Lemma \ref{S} case $(iii)$, with point $Q_\infty$ at infinity. There is a subset $S^*(U)\subseteq S(U)$ of at least $q^2-kq-(k^2+k)(k^2+k-1)-1$ points of $S(U)$, such that points of $S^*(U)$ with the same label lie on an affine line containing the point $Q_\infty$. 

Choose a point $R\in S^*(U)$ having label $v$, this label $v$ occurs at most $q$ times. Hence, there are at least
$$q^2-(k+1)q-(k^2+k)(k^2+k-1)-1$$
 points of $S^*(U)$, not with label $v$. We will call these points {\em good points}. The affine points which are not good, are called {\em bad points}.

Consider the line $\l \in\L$ defined by $R$. 
We want to find a Baer ruled cubic, containing $\l$, such that the corresponding Baer conic in $\PG(4,q)/T$ contains at least $\lfloor\frac{q+7}{2}\rfloor$ points of $S^*(U)$. Since such a conic always contains $R \in S^*(U)$, we want to find a conic with at least $\lfloor\frac{q+5}{2}\rfloor$ good points and at most $\lceil\frac{q-3}{2}\rceil$ bad points (one of which is $R$).

Consider a good point $R_1$ and its corresponding line $\l_1\in\L$. As all good points have a label different from $v$, the points $R_1$ and $R$ have a different label. Hence, the lines $\l$ and $\l_1$ intersect $T$ in a distinct point, so they are contained in a unique Baer ruled cubic (by Lemma \ref{uniquesecantplane}). Consider the corresponding Baer conic $C_1$ in $\PG(4,q)/T$.
If the conic $C_1$ contains at least $\lfloor \frac{q+5}{2}\rfloor$ good points, the result follows. Now, suppose that $C_1$ contains at most $\lfloor \frac{q+3}{2}\rfloor$ good points. Then there are at least $q^2-(k+1)q-(k^2+k)(k^2+k-1)-1-\frac{q+3}{2}$ good points that do not belong to $C_1$. Since $q\geq 4(k+1)^2$, this number is greater than zero.

Hence, we can find a good point $R_2$ that does not lie on $C_1$. The point $R_2$ defines a line $\l_2$ of $\L$. Again, we know that the lines $\l$ and $\l_2$ intersect $T$ in a different point. Take the Baer ruled cubic defined by $\l$ and $\l_2$, and consider the corresponding Baer conic $C_2$ in $\PG(4,q)/T$. Recall that two distinct Baer conics intersect in at most two points, hence $C_2$ meets $C_1$ in $R$ and in at most one other point. If the conic $C_2$ contains at least $\lfloor \frac{q+5}{2}\rfloor$ good points, the result follows.
So, suppose that at most $\lfloor\frac{q+3}{2}\rfloor$ points of $C_2$ are good points.

Since $q^2-(k+1)q-(k^2+k)(k^2+k-1)-1-2\frac{q+3}{2}>0$, we can find a good point $R_3$, not contained in $C_1\cup C_2$. Applying the same reasoning to $R_3$, we find a new Baer ruled cubic containing $\l$. The corresponding Baer conic $C_3$ contains $R$ and $R_3$, and is different from both $C_1$ and $C_2$. Thus, $C_3$ meets both in at most $1$ point different from $R$.

Continuing this reasoning, suppose we have $m=2k^2+4$ Baer conics $C_1,\ldots,C_m$ through $R$, each containing at most $\lfloor\frac{q+3}{2}\rfloor$ good points.
Hence, there are still at least
$$q^2-(k+1)q-(k^2+k)(k^2+k-1)-1 -m\frac{q+3}{2}$$
good points not contained in one of the conics $C_i$, $i=1,\ldots,m$. We obtain the parabola
$$q^2-(k^2+k+3)q-(k^4+2k^3+3k^2-k+7)$$
with largest zero point equal to
$$q= \frac{ (k^2+k+3) + \sqrt{ (k^2+k+3)^2+4(k^4+2k^3+3k^2-k+7)}}{2}.$$
Since  $$q\geq 4(k+2)^2>\frac{(k^2+k+3) + \sqrt{ (k^2+k+3)^2+4(k^4+2k^3+3k^2-k+7)}}{2},$$
there is at least one good point not on $C_1\cup\ldots \cup C_m$, say $R_{m+1}$. Consider the line $\l_{m+1}\in\L$ corresponding to $R_{m+1}$. The Baer ruled cubic $\mathcal{B}$ defined by $\l$ and $\l_{m+1}$ induces a Baer conic $C_{m+1}$ in $\PG(4,q)/T$.

There are at most $(k+1)q+(k^2+k)(k^2+k-1)+1$ bad points contained in $\PG(4,q)/T$. Each conic $C_i$, $i=1,\ldots,m$, contains at most $\lfloor \frac{q+3}{2}\rfloor$ good points, hence at least $\lceil\frac{q-1}{2}\rceil$ bad points, one of which is $R$. Since two conics have at most one bad point in common different from $R$, the conics $C_1,\ldots,C_m$ cover at least $1 + m\lceil\frac{q-3}{2}\rceil -\frac{m(m-1)}{2}$ bad points. The conic $C_{m+1}$ can intersect each conic $C_i$, $i=1,\ldots,m$, in at most one bad point. Hence, there are at most
$$1+ m + \left[(k+1)q+(k^2+k)(k^2+k-1)+1\right]-\left[1+m\frac{q-3}{2} -\frac{m(m-1)}{2}\right]$$
$$= 1+ m + (k+1)q+(k^2+k)(k^2+k-1)-m\frac{q-3}{2} +\frac{m(m-1)}{2}$$
bad points contained in $C_{m+1}$. To check that this number is strictly smaller than $ \frac{q-1}{2}$, we find the inequality
$$(-m+2k+1)q+2(k^2+k)(k^2+k-1)+m^2+4m+3<0.$$
This is equivalent to
$$q>\frac{2(k^2+k)(k^2+k-1)+m^2+4m+3}{m-2k-1},$$
which is valid when $q\geq 16$, since $$q\geq 4(k+2)^2>\frac{2(k^2+k)(k^2+k-1)+m^2+4m+3}{m-2k-1}.$$
This means that the Baer ruled cubic $\mathcal{B}$ has at most $\lceil \frac{q-3}{2}\rceil$ bad points, hence at least $\lfloor \frac{q+5}{2}\rfloor$ good points. It follows that $\mathcal{B}$ contains at least $\lfloor \frac{q+7}{2}\rfloor$ points of $S^*(U)$ and thus satisfies the conditions of the statement.
\end{proof}

\begin{lemma} \label{geencase3} Suppose $q\geq 16$ and $k\in\mathbb{N}$, $k\leq\sqrt{q}/2-2$. Let $U$ be a unital containing  a point $P_\infty$ such that $q^2-\e$, $\e\leq kq$, of the $(q+1)$-secants through $P_\infty$ are Baer secants. The corresponding affine point set $S(U)$ cannot have the form $(iii)$ of Lemma \ref{S}.
\end{lemma}
\begin{proof}
Suppose that the set $S(U)$ has the form $(iii)$ of Lemma \ref{S} with point $Q_\infty$ at infinity.
Let $\l_1$ and $\l_2$ be the lines of $\L$ defining the Baer ruled cubic $\mathcal{B}$ of Lemma \ref{Baerconic}.
A tangent subplane contains (at most) $2q+2$ points of $U$, hence $\mathcal{B}$ contains (at most) one point of $U_{\rm aff}$ not on $\l_1$ and $\l_2$.
Let $\mu$ be a plane (necessarily skew from $T$) containing a Baer conic $C$ contained in $\mathcal{B}$.  We can identify $\PG(4,q)/T$ with $\mu$, and so the intersection points of $U\cap \mathcal{B}$ define the points $R_1,R_2$ in $C$ (corresponding to $\l_1$ and $\l_2$ respectively) and at most one extra point $R$ in $C$.

By  Lemma \ref{Baerconic}, there are at least $\lfloor\frac{q+7}{2}\rfloor$ points of the Baer conic $C$ contained in $S^*(U)$, that is, two points of $S^*(U)$ with the same label lie on a line containing $Q_\infty$. Hence, we find at least two lines $L_A$ and $L_B$ through $Q_\infty$, each intersecting $C$ in two points with the same label. At most one of these lines, say $L_B$, contains the point $R$. Hence, $L_A$ intersects $C\setminus\{R\}$ in two points $Q_1, Q_2$, having the same label $v$. The points $Q_1$ and $Q_2$ are each contained in a generator line of the Baer ruled cubic, say $n_1$ and $n_2$. Since $Q_1$ and $Q_2$ are different from $R$, for $i=1,2$, the line $n_i$ either has no affine intersection point with the lines of $\mathcal{L}$ or is equal to $\l_1$ or $\l_2$.

Both points $Q_i$, $i=1,2$, have label $v$, hence, the planes $\langle T,n_i\rangle$, $i=1,2$, each contain a line of $\mathcal{L}$ through $v$, say $\l_{k_1}$ and $\l_{k_2}$ respectively. Since the line $n_i$ is either equal to $\l_{k_i}$ or does not have an affine intersection point with $\l_{k_i}$, both lines $n_i$, $i=1,2$, have to meet $T$ in $v$. This implies that we find two generator lines of the same Baer ruled cubic having a point in common, a contradiction by the definition of a ruled cubic surface, which concludes the proof.
\end{proof}

As a combination of previous lemma's, we have found that $S(U)$ must satisfy configuration $(i)$ of Lemma \ref{S}. We will show that in this case, the points of $U$ on the $q^2-\e$ Baer secants are contained in a unique unital, namely an ovoidal Buekenhout-Metz unital. This leads to the conclusion that $U$ is an ovoidal Buekenhout-Metz unital.

First, we prove that $q^2-\e$ Baer secants of an ovoidal Buekenhout-Metz unital are never contained in any other unital.
We need the definition of an {\em O'Nan configuration}, this is a collection
of four distinct lines meeting in six distinct points, as illustrated in the following picture.

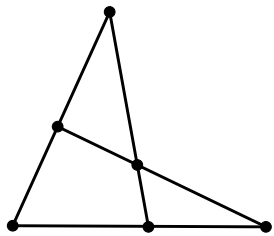
\begin{center}
\begin{tikzpicture}[line cap=round,line join=round,>=triangle 45,x=0.75cm,y=0.75cm]
\clip(9.706666666666665,2.655555555555563) rectangle (14.711111111111109,7.002222222222229);
\draw [line width=1.1pt] (11.76,6.78)-- (10.06,3.0);
\draw [line width=1.1pt] (11.76,6.78)-- (12.44,2.98);
\draw [line width=1.1pt] (14.5,2.98)-- (10.848,4.752141176470589);
\draw [line width=1.1pt] (10.06,3.0)-- (14.5,2.98);
\begin{scriptsize}
\draw [fill=black] (11.76,6.78) circle (2pt);
\draw [fill=black] (10.06,3.0) circle (2pt);
\draw [fill=black] (12.44,2.98) circle (2pt);
\draw [fill=black] (14.5,2.98) circle (2pt);
\draw [fill=black] (10.848,4.752141176470589) circle (2pt);
\draw [fill=black] (12.244110732620065,4.0746753177114075) circle (2pt);
\draw [fill=black] (0.8,2.98) circle (2pt);
\draw [fill=black] (-6.24,2.98) circle (2pt);
\end{scriptsize}
\end{tikzpicture}
\end{center}
It is known that an ovoidal Buekenhout-Metz unital contains no O'Nan configurations through its special point. A simple proof of this can be found in the proof of \cite[Lemma 7.42]{Ebert}.

We will call a line of $\PG(2,q^2)$ which is secant to a unital $U'$, a {\em $U'$-secant}.
\begin{lemma} \label{uniquenieuw} Consider an ovoidal Buekenhout-Metz unital $U'$ of $\PG(2,q^2)$ with special point $P_\infty$ and consider a set $\{L_1,\ldots,L_{\e}\}$ of $U'$-secants through $P_\infty$. Consider a unital $U$ of $\PG(2,q^2)$ containing $P_\infty$ and all points of $U'$ that do not lie on one of the $\e$ secant lines $L_i$. If $\e\leq\frac{(q-1)q}{2}$, then $U$ and $U'$ coincide.
\end{lemma}
\begin{proof}
We will show that the result holds when $\e=\frac{(q-1)q}{2}$, then the result easily follows for all $\e\leq\frac{(q-1)q}{2}$.

Consider the set $U_0$ consisting of all points contained in $U'$, but not on one of the $U'$-secants $L_i$, $i=1,\ldots,\e$. By assumption all these points are contained in $U\cap U'$.
Recall that for every unital $\widetilde{U}$, a point of $\widetilde{U}$ lies on $q^2$ $\widetilde{U}$-secants and a point not on $\widetilde{U}$ lies on only $q^2-q$ $\widetilde{U}$-secants. This means, if a point $Q$ lies on strictly more than $q^2-q$ lines intersecting $U_0$ in at least two points, then $Q$ is contained in any unital containing all points of $U_0$. Hence, in that case, $Q$ is contained in $U\cap U'$.

Consider a point $R\in U' \backslash U_0$ and say $L_1=P_\infty R$. We will prove that there are at most $q-2$ $U'$-secants $M_j$, containing $R$ but different from $L_1$, having at most 1 point in common with $U_0$. If that is the case, then there are at least $q^2-q+1$ $U'$-secants through $R$ containing at least two points of $U_0$, and hence, the point $R$ is contained in $U\cap U'$.

Consider a $U'$-secant $M_1$, different from $L_1$, containing $R$ and (at most) 1 point of $U_0$. This line intersects at least $q-1$ $U'$-secants $L_i$, different from $L_1$, in a point of $U'$, say $L_2,\ldots,L_q$.

Take a $U'$-secant $M_2$ through $R$, different from $L_1$ and $M_1$, containing at most 1 point of $U_0$. Since $U'$ contains no O'Nan configurations through the point $P_\infty$, there is at most one $U'$-secant $L_i$, $i\neq 1$, containing $P_\infty$, such that the points $L_i\cap M_1$ and $L_i \cap M_2$ are both points of $U'$.
Hence, $M_2$ intersects at least  $q-2$ new $U'$-secants $L_i$ (i.e. different from $L_1,\ldots,L_q$) in a point of $U'$, say $L_{q+1},\ldots, L_{2q-2}$.

Consider a third $U'$-secant $M_3$ through $R$, different from $L_1, M_1, M_2$. With the same reasoning as above, $M_3$ intersects at least $q-3$ $U'$-secants $L_i$ (different from $L_1,\ldots,L_{2q-2}$) in a point of $U'$, say $L_{2q-1},\ldots, L_{3q-5}$.

If there are at most $q-2$ $U'$-secants $M_j$, containing $R$ and having 0 or 1 points in common with $U_0$, the result follows. Otherwise, by continuing this process, the $U'$-secant $M_{q-1}$ intersects at least $q-(q-1)=1$ $U'$-secant $L_i$, different from the previously enumerated lines $L_1,\ldots,L_m$. We have found $m+1$ distinct $U'$-secants $L_j$ where
$$m+1 \,\, = \,\, 1+(q-1)+(q-2)+\ldots+(q-(q-2))+1 \,\, =\,\,  \frac{q(q-1)}{2}+1.$$
This is in contradiction with the restriction on the number of $U'$-secants $L_j$, since
$$\frac{q(q-1)}{2}+1 \,\, > \,\, \frac{q(q-1)}{2} \,\, = \,\, \e.$$
We have proved that there are at most $q-2$ $U'$-secants through $R$ containing 0 or 1 points of $U_0$. Hence, the point $R$ is contained in $U\cap U'$.
It follows that all points $R\in U'$ are contained in $U\cap U'$, which proves the result.
\end{proof}

\begin{lemma}\label{gedaan} Suppose $q$ and $\delta$ satisfy the conditions of Table \ref{tablecaps}. Consider a unital $U$ containing a point $P_\infty$ such that at least $q^2-\delta-1$ of the $(q+1)$-secants through $P_\infty$ are Baer secants. If $S(U)$ satisfies configuration $(i)$ of Lemma \ref{S}, then $U$ is an ovoidal Buekenhout-Metz unital with special point $P_\infty$.
\end{lemma}
\begin{proof}
If the set $S(U)$ satisfies configuration $(i)$ of Lemma \ref{S}, then all points of $S(U)$ have the same label. This implies that all $q^2-\delta-1$ lines of $\L$ go through a common point, say $v$ of the line $T$.
By Lemma \ref{uniquesecantplane}, two lines $\l_i$ and $\l_j$ of $\L$ define a unique secant subplane. By Lemma \ref{telling}, such a subplane has no affine intersection with any other line of $\L$. This means that in the 3-dimensional quotient space $\PG(4,q)/v$, the lines of $\L$ define a set $K$ of $q^2-\delta-1$ points forming a cap. As a plane through $T$ contains at most one line of $\L$, the line $T$ defines a point in this quotient space, which extends the cap $K$ to a cap $K'$ of size $q^2-\delta$. By Theorems \ref{cap} and \ref{uniquecap}, the cap $K$ can be extended to a unique ovoid $\mathcal{O}$. The cone with vertex $v$ and base $\mathcal{O}$ defines an ovoidal Buekenhout-Metz unital $U'$ which has $q^2-\delta-1$ secant lines in common with $U$. Since $\delta+1 \leq \frac{(q-1)q}{2}$, by Lemma \ref{uniquenieuw}, $U$ is an ovoidal Buekenhout-Metz unital.
\end{proof}

\begin{rtheorem}\label{main} Suppose that $q$ and $\e$ satisfy the conditions of Table \ref{tableunitals}. Let $U$ be a unital containing a point $P_\infty$ such that at least $q^2-\e$ of the $(q+1)$-secants through $P_\infty$ are Baer secants, then $U$ is an ovoidal Buekenhout-Metz unital with special point $P_\infty$.
 \end{rtheorem}
\begin{proof}
When $q$ and $\e$ satisfy the conditions of Table \ref{tableunitals}, we have $q\geq 16$ and $\e\leq \text{min}(\delta+1,\sqrt{q}q/2-2q)$ with $q$ and $\delta$ satisfying the conditions of Table \ref{tablecaps}.

Consider the set $S(U)$ defined by the Baer secants to $U$ at $P_\infty$. By Lemma \ref{SU}, this set satisfies the conditions of Lemma \ref{S}. Hence, since $q>2$ and $\e<(\sqrt{q}-1)q$, the set $S(U)$ has one of the three configurations of Lemma \ref{S}. By Lemma \ref{geendrie} ($q>2$ and $\e<(\sqrt{q}-1)q$) and Lemma \ref{geencase3} ($q\geq 16$ and $\e\leq \sqrt{q}q/{2}-2q$), only the first configuration is possible. Since $\e\leq\delta+1$, by Lemma \ref{gedaan}, $U$ is an ovoidal Buekenhout-Metz unital.
\end{proof}

Combining the Main Theorem with Theorem \ref{Barwick}, we obtain the following corollary.
\begin{corollary} Suppose that $q$ and $\e$ satisfy the conditions of Table \ref{tableunitals}. Let $U$ be a unital in $\PG(2,q^2)$. If there is a point $P_\infty$ in $U$ that lies on at least $q^2-\e$ Baer secants, and there exists a Baer secant of $U$ not through $P_\infty$, then $U$ is a classical unital.
\end{corollary}


\begin{thebibliography}{99}
\bibitem{Andre} J. Andr\'e. \"{U}ber nicht-Dessarguessche Ebenen mit transitiver Translationsgruppe. {\it Math Z.} {\bf 60} (1954), 156--186.

\bibitem{BallBlokhuisOkeefe} S. Ball, A. Blokhuis and C.M. O'Keefe. On unitals
with many Baer sublines. {\em Des. Codes Cryptogr.} {\bf 17} (1999), 237--252.


\bibitem{Ebert} S.G. Barwick and G. Ebert. {\em Unitals in projective planes. }
Springer Monographs in Mathematics. Springer, New York, 2008.

\bibitem{Barwick3}S.G. Barwick and C.T. Quinn. Generalising a characterisation of Hermitian curves. {\em J. Geom.} {\bf 70 (1--2)} (2001), 1--7.

\bibitem{Bruck-Bose} R.H. Bruck and R.C. Bose. The construction of translation planes from projective
spaces. {\it J. Algebra} {\bf 1} (1964), 85--102.

\bibitem{BuekenhoutUnitals} E. Buekenhout. Existence of unitals in finite translation planes of order $q^2$ with a kernel of order $q$. {\em Geom. Dedicata} {\bf 5} (1976), 189--194.



\bibitem{Okeefe} L.R. Casse, C.M. O'Keefe and T. Penttila. Characterizations of Buekenhout-Metz unitals. {\em Geom. Dedicata} {\bf 59 (1)} (1996), 29--42.

\bibitem{Cao}  J. Cao and L. Ou. Caps in $\PG(n,q)$ with $q$ even and $n\geq 3$. {\em Discrete Math.} {\bf 326} (2014), 61--65.

\bibitem{Chao} J.M. Chao. On the size of a cap in $\PG(n, q)$ with $q$ even and $n = 3$. {\em Geom. Dedicata} {\bf 74}
(1999), 91--94.

\bibitem{CurrentResearchTopics} F. De Clerck and N. Durante. Constructions and characterizations of classical sets in $\PG(n,q)$. In {\em Current Research Topics in Galois Geometry}. J. De Beule and L. Storme (Eds.).  Nova Science, New York, 2012, Ch. 1, 1--33.


\bibitem{PackingProblem} J. Hirschfeld and L. Storme. The packing problem in statistics, coding theory and finite projective geometries, update 2001. In {\em Finite geometries, Proceedings of the Fourth Isle of Thorns Conference}. A. Blokhuis, J.W.P. Hirschfeld, D. Jungnickel and J.A. Thas (Eds.). Dev. Math 3, Kluwer Acad. Publ., Dordrecht (2001), 201--246.

\bibitem{LinearIndependenceInFiniteSpaces} J.W.P. Hirschfeld and J.A. Thas. Linear independence in finite spaces. {\em Geom. Dedicata} {\bf 23} (1987), 15--31.
\bibitem{casse} C.T. Quinn and R. Casse. Concerning a characterisation of Buekenhout-Metz unitals. {\em J. Geom.} {\bf 52 (1--2)} (1995), 159--167.
\bibitem{Storme} L. Storme, J.A. Thas and S.K.J. Vereecke. New upper bounds for the sizes of caps in finite projective spaces. {\em J. Geom.} {\bf 73} (2002), 176--193.


\end{thebibliography}
\end{document}